\documentclass[11pt]{article}
\usepackage{amsthm}
\usepackage{amsmath}
\usepackage{amssymb}
\usepackage{framed, fullpage}

\newtheorem{theo}{Theorem}

\newtheorem{lemma}{Lemma}[section]
\newtheorem{definition}[lemma]{Definition}

\newtheorem{claim}[lemma]{Claim}

\newcommand{\tr}{\mbox{tr}}
\newcommand{\ec}{\mathsf{E}_k(T)}

\title{A Note on Even Cycles and Quasi-Random Tournaments}
\author{
Subrahmanyam Kalyanasundaram\thanks{Department of Computer Science and Engineering,
    IIT Hyderabad, India. Email: {\tt subruk@iith.ac.in}. This
    work was done while being a student in School of Computer Science,
    Georgia Institute of Technology, Atlanta, GA 30332.}
\and Asaf Shapira\thanks{School of Mathematics, Tel-Aviv University, Tel-Aviv, Israel 69978, and Schools of Mathematics and Computer Science, Georgia Institute of Technology, Atlanta, GA 30332. Email: {\tt asafico@tau.ac.il}. Supported in part by NSF Grant DMS-0901355, ISF Grant 224/11 and a Marie-Curie CIG Grant 303320.}
}

\begin{document}

\maketitle

\begin{abstract}
A cycle $C=\{v_1,v_2,\ldots,v_1\}$ in a tournament $T$ is said to be even, if when walking along $C$,
an even number of edges point in the wrong direction, that is, they are directed from $v_{i+1}$ to $v_i$.
In this short paper, we show that for every fixed even integer $k \geq 4$, if close to half of the $k$-cycles
in a tournament $T$ are even, then $T$ must be quasi-random. This resolves an open question raised in 1991
by Chung and Graham \cite{qrtour}.
\end{abstract}


\section{Introduction} \label{sec:intro}

Quasi-random (or pseudo-random) objects are {\em deterministic} objects that possess
the properties we expect truly {\em random} ones to have.
One of the most surprising phenomena in this area is the fact that in many cases,
if an object satisfies a single {\em deterministic} property then it must ``behave'' like a typical random object in many useful
aspects. In this paper we will study one such phenomenon related to quasi-random tournaments.
The notion of quasi-randomness has been widely studied for different combinatorial objects, like
graphs, hypergraphs, groups and set systems \cite{qrsetsystem, qrhyper, qrgraphs, qrgroups,T1,T2}.
We refrain from giving a detailed discussion of this area in this short paper, and instead refer the reader to the surveys
of Gowers \cite{Go2} and Krivelevich and Sudakov \cite{pseudorandom}  for more details and references.

A directed graph $D =(V,E)$ consists of a set of vertices and a set of directed edges
$E \subseteq V \times V$.
We use the ordered pair $(u,v) \in V \times V$ to denote directed edge from $u$ to $v$.
A tournament $T=(V,E)$ is a directed graph such
that given any two distinct vertices $u,v \in V$, there exists
exactly one of the two directed edges $(u,v)$ or $(v,u)$ in $E(T)$. One can also think of a tournament
as an orientation of an underlying complete graph on $V$.
We shall use $n$ to denote $|V|$.

Consider a tournament $T= (V,E)$. For $Y \subseteq V$, and $v \in V$, let $d^+(v,Y)$ denote the number of directed edges going from $v$ to $Y$
and $d^-(v,Y)$ denote the number of directed edges going from $Y$ to $v$.
A purely random tournament is one where for each pair of distinct vertices $u$ and $v$ of $V$, the directed edge
between them is chosen randomly to be either $(u,v)$ or $(v,u)$ with probability $1/2$.
It is clear that in a random tournament $T$, we have $\sum_{v\in X} \left|d^+(v,Y) - d^-(v,Y)\right| = o(n^2)$
for all $X, Y \subseteq V(T)$. Let us define the corresponding property ${\cal Q}$ as follows:

\begin{definition}\label{def:qr}
A tournament $T$ on $n$ vertices satisfies property ${\cal Q}$ if
$$ \sum_{v\in X} \left|d^+(v,Y) - d^-(v,Y)\right| = o(n^2) \;\;\mbox{ for all } X, Y \subseteq V(T).$$
\end{definition}

The notion of quasi-randomness in tournaments was introduced by Chung and Graham \cite{qrtour}.
They defined several properties of tournaments, all of which are satisfied by purely
random tournaments, including the property ${\cal Q}$ above. They also showed that all
these properties are equivalent, namely, if a tournament satisfies one of these properties,
then it must also satisfy all the other. They then defined a tournament to be quasi-random
if it satisfies any (and therefore, all) of these properties.
For the sake of brevity, we will focus on property ${\cal Q}$ (defined above) which will turn out to be the easiest
one to work with in the context of the present paper.

Another property studied in \cite{qrtour} was related to even cycles in tournaments.
A $k$-cycle is an ordered sequence of vertices $(v_1, v_2, \ldots, v_k, v_1)$ such that
no vertex is repeated immediately in the sequence.
That is, $v_i \neq v_{i+1}$ for all $i \leq k-1$ and $v_k \neq v_1$.
We say that a $k$-cycle
(for an integer $k \geq 2$)
is even if as we traverse
the cycle, we see an even number of directed edges opposite to the direction of the traversal.
If a $k$-cycle is not even, we call it odd.
Let $\ec$ denote the number of even $k$-cycles in a tournament $T$.
Clearly, the number of $k$-cycles in an $n$-vertex tournament is
$n^k - o(n^k)$. In fact, it is not hard to see that
that the exact number is given by $(n-1)^k + (-1)^k(n-1)$ (see Section
\ref{sec:conclude}).
In a random tournament, we expect about half of the $k$-cycles to be even.
This motivated Chung and Graham \cite{qrtour} to define the following property.

\begin{definition}\label{def:pk}
A tournament $T$ on $n$ vertices satisfies\footnote{Observe that our definition of a $k$-cycle allows repeated vertices in the cycle.
Note however, that forbidding repeated vertices (that is, requiring the $k$-cycles to be simple) would have resulted
in the same property ${\cal P}(k)$ since the number of $k$-cycles with repeated vertices is $o(n^k)$.
Allowing repeated vertices simplifies some of the notation.
}
property ${\cal P}(k)$ if $\ec = (1/2 \pm o(1))n^k$.
\end{definition}

Notice that when $k$ is an odd integer, $\ec$ is {\em exactly} half the number of $k$-cycles in $T$,
since an even cycle becomes odd upon traversal in the reverse direction. Hence, property ${\cal P}(k)$
cannot be equivalent to property ${\cal Q}$ when $k$ is odd.

Chung and Graham \cite{qrtour} proved that
${\cal P}(4)$ is quasi-random. In other words, a tournament has (approximately) the correct number
of even $4$-cycles we expect to find in a random tournament, if and only if it satisfies property ${\cal Q}$.
A question left open in \cite{qrtour} was whether ${\cal P}(k)$ is equivalent to ${\cal Q}$ for
all even $k \geq 4$. Our main result answers this positively by proving the following.

\begin{theo}\label{thm:main}
The following holds for every fixed even integer $k \geq 4$: A tournament satisfies property ${\cal Q}$
if and only if it satisfies property ${\cal P}(k)$.
\end{theo}

As usual, when we say that property ${\cal Q}$ implies property ${\cal P}(k)$ we mean that for every $\varepsilon$ there is a $\delta=\delta(\varepsilon)$, such
that any large enough tournament satisfying $\sum_{v\in X} \left|d^+(v,Y) - d^-(v,Y)\right| \leq \delta n^2$ for all $X,Y$ has $(1/2 \pm \varepsilon)n^k$ even cycles.
The meaning of ${\cal P}(k)$ implies ${\cal Q}$ is defined similarly.

\section{Proof of Main Result}

To prove Theorem \ref{thm:main}, we shall go through a spectral characterization of quasi-randomness.
We use the following adjacency matrix $A$ to represent the tournament $T$. For every $u,v \in V$
\[ A_{u,v} = \left\{
\begin{array}{rl}
1 & \mbox{if } (u,v) \in E(T) \\
-1 & \mbox{if } (v,u) \in E(T) \\
0 & \mbox{if } u = v
\end{array}\right. \]

A key observation that we will use is that the matrix $A$ is skew-symmetric.
Recall that a real skew
symmetric matrix can be diagonalized and all its eigenvalues are purely imaginary.
It follows that all the eigenvalues of $A^2$ are non-positive. This implies the following
claim, which will be crucial in our proof.

\begin{claim}\label{claim:eigen}
For $k \equiv 2 \pmod 4$, all the eigenvalues of $A^k$ are non-positive.
For $k \equiv 0 \pmod 4$, all the eigenvalues of $A^k$ are non-negative.
\end{claim}
%

For a matrix $M$, we let $\tr(M) = \sum_{i=1}^n M_{i,i}$ denote the trace of the matrix $M$.
Before we prove Lemmas \ref{lem:dir1} and \ref{lem:dir2}, we make the following claim.
\begin{claim}\label{claim:walk}
Let $A$ be the adjacency matrix of the tournament $T$. Then for an even integer $k \geq 4$, we have
$$ \tr(A^k) = 2 \ec - (n-1)^k - (n-1).$$
In particular, $T$ satisfies the property ${\cal P}(k)$ if and only if $|\tr(A^k)| = o(n^k)$.
\end{claim}
\begin{proof}
Notice that the $(u,u)$-th entry of $A^k$ is the number of even $k$-cycles starting and ending at $u$
minus the number of odd $k$-cycles starting and ending at $u$.
So the sum of all diagonal entries, $\tr(A^k)$, is the difference between all labeled even $k$-cycles and all labeled odd $k$-cycles.
Recall that the total number of $k$-cycles is $(n-1)^k + (n-1)$ for even $k$.
Thus we have that $\tr(A^k) = 2 \ec - (n-1)^k - (n-1)$.

We have $\tr(A^k) = 2 \ec - n^k + o(n^k)$. Notice that $T$ satisfies
property ${\cal P}(k)$ when $\ec = (1/2 \pm o(1)) n^k$, which happens if and only if
$|\tr(A^k)| = o(n^k)$.
\end{proof}
%

We are now ready to prove the first direction of Theorem \ref{thm:main}.

\begin{lemma}\label{lem:dir1}
Let $k \geq 4$ be an even integer. If a tournament satisfies
${\cal P}(k)$ then it satisfies ${\cal Q}$.
\end{lemma}

\begin{proof}
Let $ \lambda_1(A), \ldots, \lambda_n(A)$ be the eigenvalues of $A$ sorted by their
absolute value, so that $\lambda_1(A)$ has the largest absolute value.
We first claim that $|\lambda_1(A)| = o(n)$. Assume first that
$k \equiv 0 \pmod 4$. Then by Claim \ref{claim:eigen} all the eigenvalues of $A^k$ are non-negative, implying that
\begin{equation}\label{eq1}
\tr(A^k) = \sum^n_{i=1}\lambda_i(A^k) \geq \lambda_1(A^k) = \lambda_1(A)^k \;.
\end{equation}
Now, since we assume that $T$ satisfies ${\cal P}(k)$, we get from Claim \ref{claim:walk} that
$|\tr(A^k)| = o(n^k)$. Equation (\ref{eq1}) now implies that $|\lambda_1(A)| = o(n)$. If
$k \equiv 2 \pmod 4$, then since Claim \ref{claim:eigen} tells us that all eigenvalues are non-positive, we
have
\begin{equation}\label{eq3}
\tr(A^k) = \sum^n_{i=1}\lambda_i(A^k) \leq \lambda_1(A^k) = \lambda_1(A)^k \;.
\end{equation}
As in (\ref{eq1}), the fact that $|\tr(A^k)| = o(n^k)$ and that all the terms in (\ref{eq3}) are non-positive, implies
that $|\lambda_1(A)| = o(n)$.

%

We now claim that the fact that $|\lambda_1(A)| = o(n)$ implies that $T$ satisfies ${\cal Q}$.
Suppose it does not, and let $X, Y \subseteq V$ be two sets
satisfying $\sum_{v\in X} |d^+(v,Y) - d^-(v,Y)| = cn^2$, for some $c > 0$.
Let $\mathbf y \in \{0,1\}^n$ be the indicator vector for $Y$.
We pick the vector $\mathbf x$ in the following way: if $v \not \in X$, then set the corresponding coordinate
$\mathbf x_v = 0$. For $v \in X$ such that $d^+(v,Y) - d^-(v,Y) \geq 0$, we set $\mathbf x_v = 1$.
For all other $v\in X$, we set $\mathbf x_v = -1$.
Now notice that for these vectors $\mathbf x$ and $\mathbf y$, we have $\mathbf x^T A \mathbf y = \sum_{v\in X} |d^+(v,Y) - d^-(v,Y)| = cn^2$.
We can normalize $\mathbf x$ and $\mathbf y$ to get unit vectors $\tilde {\mathbf x} = \mathbf x/\sqrt{|X|}$
and $\tilde {\mathbf y} = \mathbf y/\sqrt{|Y|}$ satisfying
\begin{equation}\label{eq4}
\tilde {\mathbf x}^T A \tilde {\mathbf y} = (\mathbf x^T A \mathbf y)/ \sqrt{|X||Y|} \geq cn^2/n = cn\;,
\end{equation}
where the inequality follows since $|X|, |Y| \leq n$.
We have thus found two unit vectors $\tilde {\mathbf x}, \tilde {\mathbf y}$ such that $\tilde {\mathbf x}^T A \tilde {\mathbf y} \geq cn$.

We finish the proof by showing that (\ref{eq4}) contradicts the fact that $|\lambda_1(A)| = o(n)$.
Let $\mathbf v_1, \ldots, \mathbf v_n$ be the orthonormal eigenvectors corresponding to the eigenvalues of $A$.
Let $\tilde {\mathbf x} = \sum_i \alpha_i \mathbf v_i$ and $\tilde {\mathbf y} = \sum_i \beta_i \mathbf v_i$ be
the decomposition of $\tilde {\mathbf x}$ and $\tilde {\mathbf y}$ along the eigenvectors (note that $\alpha_i$
and $\beta_i$ might be complex numbers). We have
\begin{equation}\label{eq2}
\tilde {\mathbf x}^T A \tilde {\mathbf y}  =  \left|\sum_i \alpha_i \lambda_i(A) \beta_i\right|
                                           \leq  \sqrt{ \sum_i |\overline{\alpha_i}|^2  \cdot \sum_i |{\lambda_i(A) \beta_i}|^2}
                                           =  \sqrt{\sum_i |\lambda_i(A)|^2 |\beta_i|^2}
                                           \leq |\lambda_1(A)|
\end{equation}
where the first inequality follows by using Cauchy-Schwarz ($\overline \alpha$ denotes the complex conjugate of $\alpha$).
We then use the fact that $\sum_i |\alpha_i|^2 =\sum_i |\beta_i|^2=1$ which follow from the fact that $\tilde {\mathbf x},\tilde {\mathbf y}$
are unit vectors. Finally, since we have that $|\lambda_1(A)| = o(n)$ and that $\tilde {\mathbf x}^T A \tilde {\mathbf y} \geq cn$ equation (\ref{eq2}) gives a contradiction.
So $T$ must satisfy ${\cal Q}$.
\end{proof}

We now turn to prove the second direction of Theorem \ref{thm:main}.

\begin{lemma}\label{lem:dir2}
Let $k \geq 4$ be an even integer. If a tournament satisfies ${\cal Q}$
then it satisfies ${\cal P}(k)$.
\end{lemma}

\begin{proof}
Suppose $T$ satisfies ${\cal Q}$. Then by the result of $\cite{qrtour}$ mentioned earlier, $T$ must also satisfy ${\cal P}(4)$.
From Claim \ref{claim:walk}, we have that
\begin{equation}\label{eq5}
|\tr(A^4)| = \left| \sum^n_{i=1}\lambda^4_i \right|= o(n^4)\;,
\end{equation}
where $\lambda_1,\ldots,\lambda_n$ are the eigenvalues of $A$.
We will now apply induction to show that $|\tr(A^k)| = o(n^k)$ for all even integers $k \geq 4$.
Claim \ref{claim:walk} would then imply that ${\cal P}(k)$ is true for all even integers $k \geq 4$.

Now note the following for an even integer $k > 4$:
$$
|\tr(A^k)| = \left|\sum_i \lambda_i^{k}\right| \leq \sqrt{\sum_{i}\lambda_i^{4} \sum_i \lambda_i^{2k-4}}
\leq  \sqrt{\sum_{i}\lambda_i^{4}} \cdot \left|\sum_{i} \lambda_i^{k-2}\right|
= o(n^k)\;.
$$
The first inequality is Cauchy-Schwarz. For the second inequality, recall that by Claim \ref{claim:eigen} we have that $\lambda^k_i$ are either all non-negative or non-positive.
This means that $(\sum^n_{i=1}\lambda^{k-2}_i)^2 \geq \sum^n_{i=1}\lambda^{2k-4}_i$ since we lose only non-negative terms.
The last equality follows by applying the induction hypothesis and (\ref{eq5}).
\end{proof}

\section{Concluding Remarks}\label{sec:conclude}

\begin{itemize}

\item The proof of Lemma \ref{lem:dir1} shows that if $T$ satisfies the property ${\cal P}(4)$,
then $|\lambda_1(A)| = o(n)$ which in turn implies that $T$ satisfies ${\cal Q}$. Since
we also know that ${\cal Q}$ implies ${\cal P}(4)$ we conclude that a tournament $T$ is quasi-random if and only if
$|\lambda_1(A)| = o(n)$. This is in line with other spectral characterizations
of quasi-randomness for other combinatorial objects \cite{eigenexp, qrregular, butler, qrgraphs, cayley}.

\item Let $k\geq 4$ be an even integer. Now we make an observation about $\ec$
for an arbitrary tournament $T$ (which is not necessarily quasi-random).
The total number of distinct $k$-cycles of $T$ is  $\tr(B^k)$,
where $B$ is the adjacency matrix of the undirected complete graph on $n$ vertices.
Since the spectrum of $B$ is $\{ n-1,-1,\ldots,-1 \}$ we get $\tr(B^k) = (n-1)^{k} + (n-1)$.
For $k \equiv 0 \pmod 4$, by Claim \ref{claim:eigen}, the eigenvalues of $A^k$ are all non-negative
and thus we have
$\tr(A^k) \geq 0$. By Claim \ref{claim:walk}, we have that $\ec \geq ((n-1)^k + (n-1))/2$.
For $k \equiv 2 \pmod 4$, we can conclude similarly using Claims \ref{claim:eigen} and \ref{claim:walk}
that $\ec \leq ((n-1)^k + (n-1))/2$.

\item We note that we can use the ideas we used in this paper to prove similar results for general directed graphs
as defined by Griffiths \cite{qrdigraphs}. Since the ideas required to obtain this more general result do not deviate significantly from those
we have used here, we defer them to the first author's Ph.D. thesis.

\end{itemize}

\noindent \textbf{Acknowledgement:} The first author would like to thank Pushkar Tripathi for helping with
computer simulations.


\end{document}